\documentclass{article}%
\usepackage{amsmath}
\usepackage{amsfonts}
\usepackage{amssymb}
\usepackage{graphicx}%
\providecommand{\U}[1]{\protect\rule{.1in}{.1in}}
\newtheorem{theorem}{Theorem}

\newtheorem{corollary}[theorem]{Corollary}

\newtheorem{lemma}[theorem]{Lemma}

\newtheorem{proposition}[theorem]{Proposition}
\newtheorem{remark}[theorem]{Remark}

\newenvironment{proof}[1][Proof]{\noindent\textbf{#1.} }{\ \rule{0.5em}{0.5em}}
\begin{document}

\title{Null-orbit reflexive operators}
\author{Don Hadwin\\University of New Hampshire
\and Ileana Iona\c{s}cu\\Philadelphia University
\and Hassan Yousefi\\California State University Fullerton}
\maketitle

\begin{abstract}
We introduce and study the notion of null-orbit reflexivity, which is a slight
perturbation of the notion of orbit-reflexivity. Positive results for orbit
reflexivity and the recent notion of $\mathbb{C}$-orbit reflexivity both
extend to null-orbit reflexivity. Of the two known examples of operators that
are not orbit-reflexive, one is null-orbit reflexive and the other is not. The
class of null-orbit reflexive operators includes the classes of hyponormal ,
algebraic, compact, strictly block-upper (lower) triangular operators, and
operators whose spectral radius is not 1. We also prove that every
polynomially bounded operator on a Hilbert space is both orbit-reflexive and
null-orbit reflexive.

\end{abstract}

\bigskip

\bigskip

\section{Introduction\bigskip}

In a recent paper \cite{MV} the authors and M. McHugh introduced a new notion
of reflexivity for operators, $\mathbb{C}$\emph{-orbit reflexivity} as well as
its linear-algebraic analogue. This notion is related to the notion of orbit
reflexivity \cite{HNRR}. Examples of Hilbert space operators that are not
orbit reflexive can be found in two very remarkable papers; the first example
was given by S. Grivaux and M. Roginskaya \cite{GR}, and the second, much
simpler, example was given by V. M\"{u}ller and J. Vr\v{s}ovsk\'{y} \cite{MV}.

Although even in finite-dimensions there is an ample supply of operators that
are not $\mathbb{C}$-orbit reflexive, it was easy to show that operators that
are strictly block-upper(or lower)-triangular are $\mathbb{C}$-orbit
reflexive. This fact combined with the example of a non-orbit-reflexive
operator in \cite{MV}, led us naturally to a new version of orbit reflexivity,
null-orbit reflexivity, that includes all of the previously-proved
orbit-reflexive operators but excludes the counterexample in \cite{MV}.

Suppose $T$ is a linear transformation on a vector space. We define the
\emph{null-orbit} of $T$ as%
\[
\mathrm{nullOrb}\left(  T\right)  =\left\{  0,1,T,T^{2},\ldots\right\}  .
\]
The \emph{orbit} of $T$ is \textrm{Orb}$\left(  T\right)  =\left\{
1,T,T^{2},\ldots\right\}  $. We define \textrm{null}$\mathrm{OrbRef}%
_{0}\left(  T\right)  $ to be the set of all linear transformations $S$ such
that for every vector $x$%

\[
Sx\in\mathrm{null}\text{-}\mathrm{Orb}\left(  T\right)  x
\]
and we say that $T$ is \emph{algebraically} \emph{null-orbit reflexive} if%
\[
\mathrm{nullOrbRef}_{0}\left(  T\right)  =\mathrm{\mathrm{nullOrb}}\left(
T\right)  \mathrm{.}%
\]
If $T$ is a bounded operator on a Banach space, we define $\mathrm{nullOrbRef}%
\left(  T\right)  $ to be the set of all operators $S$ such that, for every
vector $x$%
\[
Sx\in\left[  \mathrm{nullOrb}\left(  T\right)  x\right]  ^{-},
\]
and we say that $T$ is \emph{null-orbit reflexive} if $\mathrm{nullOrbRef}%
\left(  T\right)  $ is the strong-operator closure of $\mathrm{nullOrb}\left(
T\right)  $. Orbit reflexivity is defined as in the above definition replacing
$\mathrm{nullOrb}\left(  T\right)  $ with $\mathrm{Orb}\left(  T\right)  $.
The slight change in definitions causes drastic changes in the two notions.

In this paper we extend all of the positive known results for orbit
reflexivity to null-orbit reflexivity, and we show that most of the positive
results for $\mathbb{C}$-orbit reflexivity extend to null orbit reflexivity.
Moreover, for the example in \cite{MV} of a Hilbert space operator $T$, that
is not orbit reflexive, we show that $T$ is null-orbit reflexive. In the
example in \cite{GR} of a Hilbert space operator that is not orbit reflexive,
the proof shows that the operator is also not null-orbit reflexive.

We first prove a number of results in the purely algebraic case, and we use
these to prove several results for operators on a normed space or a Hilbert
space. We next extend the results of \cite{HNRR} and \cite{MV} to the
null-orbit reflexivity case. We finish with a new result that every
polynomially bounded operator on a Hilbert space is both orbit-reflexive and
null-orbit reflexive.

Suppose $X$ is a normed space and $\mathcal{A}$ is an algebra of (bounded
linear) operators on $X$. A (closed linear) subspace $M$ of $X$ is
$\mathcal{A}$\emph{-invariant} if $A\left(  M\right)  \subseteq M$ for every
$A\in\mathcal{A}$. We let \textrm{Lat}$\mathcal{A}$ denote the set of all
invariant subspaces for $\mathcal{A}$, and we let \textrm{AlgLat}$\mathcal{A}$
denote the algebra of all operators that leave invariant every $\mathcal{A}%
$-invariant subspace. The algebra $\mathcal{A}$ is \emph{reflexive} if
$\mathcal{A}=$\textrm{AlgLat}$\mathcal{A}$. If the algebra $\mathcal{A}$
contains the identity operator $1$, then $S\in$\textrm{AlgLat}$\mathcal{A}$ if
and only if, for every $x\in X,$ $Sx$ is in the closure of $\mathcal{A}x$.
This characterization works equally well for a linear subspace $\mathcal{S}$
of $B\left(  X\right)  $ (the set of all operators on $X$), i.e., we define
\textrm{ref}$\mathcal{S}$ to be the set of all operators $A$ such that, for
every $x\in X$, we have $Ax$ is in the closure of $\mathcal{S}x$, and we say
that $\mathcal{S}$ is \textrm{reflexive} if $\mathcal{S}=\mathrm{ref}%
\mathcal{S}$. If we let $T$ be a single operator and let $\mathcal{S}%
=\mathrm{Orb}\left(  T\right)  =\left\{  T^{n}:n\geq0\right\}  ,$ we apply the
same process to obtain the notion of orbit reflexivity. (Note that in this
case $\mathcal{S}$ is not a linear space.) We define \textrm{OrbRef}$\left(
T\right)  $ to be the set of all operators $A$ such that, for every vector
$x,$ we have $Ax$ is in the closure of Orb$\left(  T,x\right)  =$ Orb$\left(
T\right)  x$. We say that $T$ is \emph{orbit reflexive} if \textrm{OrbRef}%
$\left(  T\right)  $ is the closure of \textrm{Orb}$\left(  T\right)  $ in the
strong operator topology (SOT).

\section{Algebraic Results}

Throughout this section $\mathbb{F}$ will denote an arbitrary field, $X$ will
denote a vector space over $\mathbb{F}$, and $\mathcal{L}\left(  X\right)  $
will denote the algebra of all linear transformations on $X$.

A transformation $T\in\mathcal{L}\left(  X\right)  $ is \emph{locally
nilpotent} if $X=\cup_{n\geq1}\ker\left(  T^{n}\right)  $. More generally $T$
is \emph{locally algebraic} if, for each $x\in X$, there is a nonzero
polynomial $p_{x}\in\mathbb{F}\left[  t\right]  $ such that $p_{x}\left(
T\right)  x=0$. If $p_{x}\left(  t\right)  $ is chosen to be monic with
minimal degree, we call $p_{x}$ a \emph{local polynomial} for $T$ at $x$.

\bigskip

\begin{theorem}
\label{locnil}Every locally nilpotent linear transformation on a vector space
$X$ over field $\mathbb{F}$ is algebraically null-orbit reflexive. Moreover,
if $S\in\mathrm{nullOrbRef}_{0}\left(  T\right)  ,$ $x\in X$, and
$Sx=T^{k}x\neq0,$ then $S=T^{k}.$
\end{theorem}

\begin{proof}
We know from \cite[Theorem 1]{HIMY} that $T$ is algebraically $\mathbb{F}%
$-orbit reflexive. Thus if $S\in\mathrm{nullOrbRef}_{0}\left(  T\right)  $ and
$S\neq0,$ then there is an $x\in X$ and an integer $n\geq0$ such that
$Sx=T^{n}x\neq0$, and it follows from \cite[Theorem 1]{HIMY} that $S=T^{n}%
$.\bigskip
\end{proof}

For infinite fields the next theorem reduces the problem of algebraic
null-orbit reflexivity to the case of locally algebraic transformations. A key
ingredient in the proof is an algebraic reflexivity result from \cite{H1} that
says if $\mathbb{F}$ is infinite and $T\in\mathcal{L}\left(  X\right)  $ is
not locally algebraic, then, whenever $S\in\mathcal{L}\left(  X\right)  $ and
for every $x\in X$ there is a polynomial $p_{x}$ such that $Sx=p_{x}\left(
T\right)  x$, we must have $S=p\left(  T\right)  $ for some polynomial $p$.

\begin{theorem}
\label{notlocalg} Suppose $X$ is a vector space over an infinite field
$\mathbb{F}$, and suppose $T\in\mathcal{L}\left(  X\right)  $ is not locally
algebraic. Then $T$ is algebraically null-orbit reflexive.
\end{theorem}

\begin{proof}
Suppose $S\in\mathrm{nullOrbRef}_{0}\left(  T\right)  $. Then $Sx\in
\mathrm{nullOrb}\left(  T\right)  x$ for every $x\in X$. It follows from
\cite{H1} that $T$ is algebraically reflexive, so we know there is a
polynomial $p\in\mathbb{F}\left[  t\right]  $ such that $S=p\left(  T\right)
.$ Since $T$ is not locally algebraic, there is a vector $e\in X$ such that
for every nonzero polynomial $q\in\mathbb{F}\left[  t\right]  ,$ we have
$q\left(  T\right)  e\neq0$. Since $S\in\mathrm{nullOrbRef}_{0}\left(
T\right)  ,$ we know that there is an $n\geq0$ such that $Se=T^{n}e.$ Hence
$p\left(  t\right)  =t^{n},$ and thus $S\in\mathrm{nullOrb}\left(  T\right)  $.
\end{proof}

\bigskip

\begin{remark}
If there is an $A\in\mathrm{OrbRef}_{0}\left(  T\right)  $ such that $AT\neq
TA,$ then, since $\mathrm{OrbRef}_{0}\left(  T\right)  \subseteq
\mathrm{nullOrbRef}_{0}\left(  T\right)  $, it follows that $T$ is not
algebraically null-orbit reflexive. Similarly, if $T$ acts on a Banach space,
and there is an $A\in\mathrm{OrbRef}\left(  T\right)  $ such that $AT\neq TA$,
then $T$ is not null-orbit reflexive. Hence the Hilbert space operator
constructed by S. Grivaux and M. Roginskaya \cite{GR} is not null-orbit reflexive.
\end{remark}

\bigskip

The preceding remark naturally leads to a pair of questions.

\bigskip

\textbf{Question 1.} If $S\in\mathrm{nullOrbRef}_{0}\left(  T\right)  $ and
$ST=TS,$ must $S\in\mathrm{nullOrb}\left(  T\right)  $?

\bigskip

\textbf{Question 2.} If $T$ acts on a Hilbert space, $S\in\mathrm{nullOrbRef}%
\left(  T\right)  $ and $ST=TS$, must $S$ be in the strong-operator closure of
$\mathrm{nullOrb}\left(  T\right)  $? What is the answer if we assume that $S$
is in the double commutant of $\left\{  T\right\}  $?

\bigskip

Note that the example of V. M\"{u}ller and J. Vr\v{s}ovsk\'{y} \cite[Example
1]{MV}, where $S=0\in\mathrm{OrbRef}\left(  T\right)  \backslash
\mathrm{Orb}\left(  T\right)  ^{-SOT}$ shows that the analog of Question 2 for
orbit reflexivity has a negative answer. We will see later (Corollary
\ref{ex}) that their example is null-orbit reflexive, so it has no bearing on
Question 2. In \cite{MV} an example is given of an operator on $\ell^{1}$ that
is reflexive but not orbit reflexive. In view of Theorem 2.8 and Proposition
3.1 in \cite{HN}, it seems feasible that the operator $T$ in Example 1 of
\cite[Example 1]{MV} is reflexive. We know that \textrm{AlgLat}$T\subseteq
\left\{  T\right\}  ^{\prime\prime}$ and that if $S\in$\textrm{AlgLat}$T$,
then there is a sequence $\left\{  a_{n}\right\}  _{n\geq0}$ such that, for
every vector $x,$ $Sx \sim %
{\displaystyle\sum_{n=0}^{\infty}}
a_{n}T^{n}$ in the sense of \cite{HN}. 

\bigskip

\textbf{Question 3.} Is the operator in Example 1 of \cite[Example 1]{MV} is reflexive?

\bigskip

The proof of Theorem \ref{notlocalg} shows that if $T$ is algebraically
$\mathbb{F}$-orbit reflexive (reflexive) and $\mathbb{F}$-Orb$\left(
T\right)  $ ($\left\{  p\left(  T\right)  :p\in\mathbb{F}\left[  t\right]
\right\}  $) has a separating vector, then $T$ is algebraically null-orbit
reflexive. This immediately gives us the following (see \cite[Theorem 3]{HIMY}).

\bigskip

\begin{theorem}
\label{findim}Suppose $X$ is a finite-dimensional vector space over a field
$\mathbb{F}$ not isomorphic to $\mathbb{Z}/p\mathbb{Z}$ for some prime $p$.
Then every linear transformation on $X$ whose minimal polynomial splits over
$\mathbb{F}$ is algebraically null-orbit reflexive.
\end{theorem}

\bigskip

\begin{corollary}
If $X$ is a finite-dimensional vector space over an algebraically closed field
$\mathbb{F}$, then every linear transformation on $X$ is algebraically
null-orbit reflexive. \bigskip
\end{corollary}

Recall from ring theory that if $\mathcal{R}$ is a principal ideal domain, $M$
is an $\mathcal{R}$-module, $0\neq r\in\mathcal{R}$ and $rM=\left\{
0\right\}  ,$ then $M$ is a direct sum of cyclic $\mathcal{R}$-modules;
Applying this fact to $\mathcal{R}=\mathbb{F}\left[  t\right]  $, we get that
any algebraic linear transformation on a vector space is a direct sum of
transformations on finite-dimensional subspaces, and therefore has a Jordan
form when the minimal polynomial splits over $\mathbb{F}$. (See \cite{Kap} for
details.) This gives us the following corollary.

\bigskip

\begin{corollary}
\label{algTrans}Suppose $X$ is a vector space over a field $\mathbb{F}$ not
isomorphic to $\mathbb{Z}/p\mathbb{Z}$ for some prime $p$. Then every
algebraic linear transformation on $X$ whose minimal polynomial splits over
$\mathbb{F}$ is algebraically null-orbit reflexive.
\end{corollary}

\bigskip

\bigskip\ 

\section{Null-orbit reflexivity}

The following result was proved in \cite[Proposition 3]{HNRR}.

\bigskip

\begin{lemma}
\label{normal}Suppose $\mathcal{N}$ is a commuting family of normal operators
on a Hilbert space $X$ and $A\in B\left(  X\right)  $ satisfies, for every
$x\in X$, $Ax\in\left(  \mathcal{N}x\right)  ^{-}$. Then $A$ is in the
SOT-closure of $\mathcal{N}$.\bigskip
\end{lemma}

If in the preceding lemma we let $\mathcal{N}=\left\{  0,1,T,T^{2}%
,\ldots\right\}  $, we obtain the following. \bigskip

\begin{proposition}
\label{nor}Every normal operator on a Hilbert space is null-orbit reflexive.
\end{proposition}

\bigskip

The next two results are consequences of Theorem \ref{locnil}.

\begin{theorem}
\label{unionker}Suppose $T$ is a bounded linear operator on a real or complex
normed space $X$ such that $\cup_{n=1}^{\infty}\ker\left(  T^{n}\right)  $ is
dense in $X$. Then $T$ is null-orbit reflexive and $\mathrm{nullOrb}\left(
T\right)  $ is SOT-closed. Moreover, if $S\in\mathrm{nullOrbRef}\left(
T\right)  $, $x\in\cup_{n=1}^{\infty}\ker\left(  T^{n}\right)  ,$ $k\geq0,$
and $Sx=T^{k}x\neq0$, then $S=T^{k}$.
\end{theorem}

\begin{proof}
Suppose $S\in\mathrm{nullOrbRef}\left(  T\right)  ,$ and let $M=\cup
_{n=1}^{\infty}\ker\left(  T^{n}\right)  $. It is clear that $S\left(
M\right)  \subseteq M$ and $T\left(  M\right)  \subseteq M$ and $S|_{M}%
\in\mathrm{nullOrbRef}_{0}\left(  T|M\right)  $. But $T|M$ is locally
nilpotent, and if $x\in M$ and $T^{n}x=0$, then
\[
\mathrm{nullOrb}\left(  T\right)  x=\left\{  0\right\}  \cup\left\{
x,Tx,\ldots,T^{n-1}x\right\}
\]
is norm closed. Hence, $\mathrm{nullOrbRef}\left(  T|M\right)
=\mathrm{nullOrbRef}_{0}\left(  T|M\right)  ,$ which, by Theorem \ref{locnil}
is $\mathrm{nullOrb}\left(  T|M\right)  $. Hence there is an $A\in
\mathrm{nullOrb}\left(  T\right)  $ such that $S|M=A|M$. However, $M$ is dense
in $X,$ so $S=A\in\mathrm{nullOrb}\left(  T\right)  $.
\end{proof}

\bigskip

The preceding theorem implies a stronger version of itself.

\begin{corollary}
Suppose $X$ is a real or complex normed space, and there is a decreasingly
directed family $\left\{  X_{\lambda}:\lambda\in\Lambda\right\}  $ of
$T$-invariant closed linear subspaces such that
\end{corollary}

\begin{enumerate}
\item for every $\lambda\in\Lambda$, $\cup_{n=0}^{\infty}\left(  T^{n}\right)
^{-1}\left(  X_{\lambda}\right)  $ is dense in $X$, and

\item $\cap_{\lambda\in\Lambda}X_{\lambda}=\left\{  0\right\}  .$
\end{enumerate}

Then $T$ is null-orbit reflexive and $\mathrm{nullOrbRef}\left(  T\right)
=\mathrm{nullOrb}\left(  T\right)  $.

\begin{proof}
Suppose $S\in\mathrm{nullOrbRef}\left(  T\right)  $ and $S\neq0$. Choose $e\in
X$ such that $Se\neq0$. It follows from $\left(  2\right)  $ that both
$\left(  1\right)  $ and $\left(  2\right)  $ remain true if we consider only
those $X_{\lambda}$ that contain neither $e$ nor $Se$. Since $T\left(
X_{\lambda}\right)  \subseteq X_{\lambda}$, $\hat{T}_{\lambda}\left(
x+X_{\lambda}\right)  =Tx+X_{\lambda}$ defines a bounded linear operator
$\hat{T}_{\lambda}$ on $X/X_{\lambda}$. Condition $\left(  1\right)  $ implies
that $\cup_{n=1}^{\infty}\ker\left(  \hat{T}_{\lambda}^{n}\right)  $ is dense
in $X/X_{\lambda};$ whence, by Theorem \ref{unionker}, $\hat{T}_{\lambda}$ is
null-orbit reflexive. However, $S\in\mathrm{nullOrbRef}\left(  T\right)  $
implies that $S\left(  X_{\lambda}\right)  \subseteq X_{\lambda},$ so $\hat
{S}_{\lambda}\left(  x+X_{\lambda}\right)  =Sx+X_{\lambda}$ defines an
operator on $X/X_{\lambda}$ such that $\hat{S}_{\lambda}\in
\mathbb{\mathrm{nullOrbRef}}\left(  \hat{T}_{\lambda}\right)  .$ Hence, by
Theorem \ref{unionker}, there is a unique nonnegative integer $n_{\lambda}$
such that $\hat{S}_{\lambda}=\hat{T}_{\lambda}^{n_{\lambda}}.$ Suppose
$\eta\in\Lambda$. Since the $X_{\lambda}$'s are decreasingly directed, there
is a $\sigma\in\Lambda$ such that $X_{\sigma}\subseteq X_{\lambda}\cap
X_{\eta}$. Applying the same arguments we used on $X_{\lambda}$, there is a
unique integer $m\geq0$ such that $\hat{S}_{\sigma}=T_{\sigma}^{n_{\sigma}}$.
However, it follows from $\left(  1\right)  $that there is a vector
$x\in\left[  \cup_{n=0}^{\infty}\left(  T^{n}\right)  ^{-1}\left(  X_{\sigma
}\right)  \right]  \backslash X_{\lambda}$. Then there is an $n$ such that
$T^{n}x\in X_{\sigma}\subseteq X_{\lambda}$ and thus $\hat{T}^{n}\left(
x+X_{\lambda}\right)  =0$ but $x+X_{\lambda}\neq0.$ However, $Sx-T^{n_{\sigma
}}x\in X_{\sigma}\subseteq X_{\lambda},$ so
\[
\hat{S}_{\lambda}\left(  x+X_{\lambda}\right)  =\bar{T}_{\lambda}^{n_{\sigma}%
}\left(  x+X_{\lambda}\right)  =\bar{T}_{\lambda}^{n_{\lambda}}\left(
x+X_{\lambda}\right)  ,
\]
which implies that $n_{\sigma}=n_{\lambda}$. Similarly, $n_{\sigma}%
=n_{\lambda}.$ Hence there is an integer $n\geq0$ such that, for every
$\lambda\in\Lambda$, $n_{\lambda}=n$. Hence, for every $x\in X$ and every
$\lambda\in\Lambda,$%
\[
Sx-T^{n}x\in X_{\lambda,}%
\]
which, by $\left(  2\right)  $, implies $S=T^{n}$.
\end{proof}

\bigskip

The following corollary applies to operators that have a strictly
upper-triangular operator matrix with respect to some direct sum decomposition.

\begin{corollary}
If a normed space $X$ over $\mathbb{F}\in\left\{  \mathbb{R},\mathbb{C}%
\right\}  $ is a direct sum of spaces $\left\{  X_{n}:n\in\mathbb{N}\right\}
$ such that $T\left(  X_{1}\right)  =\left\{  0\right\}  ,$ and for every
$n>1$,
\[
T\left(  X_{n}\right)  \subseteq\left(
{\displaystyle\sum\nolimits_{k<n}^{\oplus}}
X_{k}\right)  ^{-},
\]
then $T$ is null-orbit reflexive and $\mathrm{nullOrbRef}\left(  T\right)
=\mathrm{nullOrb}\left(  T\right)  $.
\end{corollary}

\bigskip

The preceding corollary has some familiar special cases.\bigskip

\begin{corollary}
If $T$ is an operator-weighted (unilateral, bilateral, or backwards) shift or
if $T$ is a direct sum of nilpotent operators on a real or complex normed
space $X$, then $T$ is null-orbit reflexive.
\end{corollary}

\bigskip

\bigskip

\begin{theorem}
Suppose $X$ is a normed space over $\mathbb{F}\in\left\{  \mathbb{R}%
,\mathbb{C}\right\}  $, $T\in B\left(  X\right)  $ and $\cap_{n=1}^{\infty
}T^{n}\left(  X\right)  ^{-}=\left\{  0\right\}  $. Then $T$ is null-orbit
reflexive and $\mathrm{nullOrbRef}\left(  T\right)  =\mathrm{nullOrb}\left(
T\right)  $. Moreover, if $S\in\mathrm{nullOrbRef}\left(  T\right)  ,$ $x\in
X,$ and $0\neq Sx=T^{k}x$, then $S=T^{k}$.
\end{theorem}

\begin{proof}
We will first show that $T$ is algebraically null-orbit reflexive. If $M$ is a
finite-dimensional invariant subspace for $T$ and $T|M$ is not nilpotent, then
there is a nonzero $T$-invariant subspace $N$ of $M$ such that $\ker\left(
T|N\right)  =0.$ Thus $T\left(  N\right)  =N\neq0,$ which violates $\cap
_{n=1}^{\infty}T^{n}\left(  X\right)  ^{-}=\left\{  0\right\}  $. Thus, either
$T$ is not locally algebraic or $T$ is locally nilpotent. In these cases it
follows either from Theorem \ref{notlocalg} or Theorem \ref{locnil} that $T$
is indeed algebraically null-orbit reflexive. Furthermore, the hypothesis on
$T$ implies, for each $x\in X,$ that
\[
\cap_{N=1}^{\infty}\left\{  T^{k}x:k\geq N\right\}  ^{-}=\left\{  0\right\}
\text{,}%
\]
so $\mathbb{\ }\mathrm{nullOrb}\left(  T\right)  x$ is closed in $X.$ Thus
$\mathrm{nullOrbRef}\left(  T\right)  =\mathrm{nullOrbRef}_{0}\left(
T\right)  =\mathrm{nullOrb}\left(  T\right)  $. For the last statement suppose
$x\in X,$ and $k,n\geq0$ are integers, and
\[
0\neq Sx=T^{n}x=T^{k}x.
\]
Suppose $k<n$. Then $M=sp\left\{  x,Tx,\ldots,T^{n-1}x\right\}  $ is a nonzero
finite-dimensional invariant subspace for $T$ with $\dim M\leq n$. Since
$T^{n}x\neq0$, we know $T|M$ is not nilpotent, which, as remarked earlier,
contradicts $\cap_{n=1}^{\infty}T^{n}\left(  X\right)  ^{-}=\left\{
0\right\}  $.
\end{proof}

\bigskip

This theorem also implies a stronger version of itself.\bigskip

\begin{corollary}
Suppose $X$ is a real or complex normed space, $T\in B\left(  X\right)  ,$ and
there is an increasingly directed family $\left\{  X_{\lambda}:\lambda
\in\Lambda\right\}  $ of $T$-invariant linear subspaces such that
\end{corollary}

\begin{enumerate}
\item for every $\lambda\in\Lambda,$ $\cap_{n=1}^{\infty}\overline
{T^{n}\left(  X_{\lambda}\right)  }=\left\{  0\right\}  $, and

\item $\cup_{\lambda\in\Lambda}X_{\lambda}$ is dense in $X.$
\end{enumerate}

Then $T$ is null-orbit reflexive, and $\mathrm{nullOrbRef}\left(  T\right)
=\mathrm{nullOrb}\left(  T\right)  $. Moreover, if $S\in\mathrm{nullOrbRef}%
\left(  T\right)  ,$ $x\in X,$ and $0\neq Sx=T^{k}x$, then $S=T^{k}$.

\begin{proof}
Suppose $0\neq S\in\mathrm{nullOrbRef}\left(  T\right)  $. It follows from
$\left(  2\right)  $ that there is a $\lambda_{0}\in\Lambda$ and an $f\in
X_{\lambda_{0}}$ such that $0\neq Sf.$ However, we must have $S\left(
X_{\lambda_{0}}\right)  \subseteq X_{\lambda_{0}},$ and $S|X_{\lambda_{0}}%
\in\mathrm{nullOrbRef}\left(  T|X_{\lambda_{0}}\right)  =\mathrm{nullOrb}%
\left(  T|X_{\lambda}\right)  $ (by $\left(  1\right)  $ and the preceding
theorem). Thus there is an integer $k\geq0$ such that%
\[
S|X_{_{\lambda_{0}}}=T^{k}|X_{\lambda_{0}}.
\]
The same $k$ must work for any $X_{\lambda}$ that contains $X_{\lambda_{0}}.$
It follows from the fact that the family is increasingly directed and $\left(
2\right)  $ that $S=T^{k}$.\bigskip
\end{proof}

If $T$ is the operator constructed in \cite{MV} that is not orbit reflexive,
it is easy to show that $\cap_{n\geq0}T^{n}\left(  X\right)  ^{-}=0.$

\begin{corollary}
\label{ex}The non orbit reflexive operator constructed in \cite{MV} is
null-orbit reflexive.
\end{corollary}

\bigskip

Irving Kaplansky \cite{Kap} (see also \cite{Lar}, \cite{Liv} , \cite{M})
proved that a (bounded linear) operator on a Banach space is locally algebraic
if and only if it is algebraic. This immediately gives us the following result
from Theorem \ref{notlocalg}.

\begin{proposition}
\label{algebraic}Suppose $X$ is a real or complex Banach space and $T\in
B\left(  X\right)  $ is not algebraic. Then $T$ is algebraically null-orbit reflexive.
\end{proposition}

\ \bigskip

\bigskip The results in the paper of \cite{MV} also extend to the null-orbit
case. If $T$ is an operator on a Banach space, then $r\left(  T\right)  $
denotes the spectral radius of $T$, i.e.,%
\[
r\left(  T\right)  =\max\left\{  \left\vert \lambda\right\vert :\lambda
\in\sigma\left(  T\right)  \right\}  \text{.}%
\]

\begin{lemma}
\label{cat}If $X$ is a normed space, $T\in B\left(  X\right)  $ and
\[
E=\left\{  x\in X:\mathrm{nullOrb}\left(  T\right)  x\text{ is norm
closed}\right\}
\]
is not contained in a countable union of nowhere dense subsets of $X$, then
$T$ is null-orbit reflexive and $\mathrm{nullOrbRef}\left(  T\right)
=\mathrm{nullOrb}\left(  T\right)  $. (Note that $E$ contains all $x\in X$
such that $T^{n}x\rightarrow0$ weakly or $\left\Vert T^{n}x\right\Vert
\rightarrow\infty$.)
\end{lemma}

\begin{proof}
If $S\in\mathrm{\mathrm{nullOrbRef}\left(  T\right)  }$, then $E\subseteq
\cup_{A\in\mathrm{nullOrb}\left(  T\right)  }\ker\left(  S-A\right)  ,$ so
there is an $A\in\mathrm{nullOrb}\left(  T\right)  $ such that $\ker\left(
S-A\right)  $ has nonempty interior, which means that $S=A$.
\end{proof}

\bigskip

\begin{corollary}
If $X$ is a Banach space, $T\in B\left(  X\right)  $ and $r\left(  T\right)
<1$, then $T$ is null-orbit reflexive.
\end{corollary}

\begin{proof}
It follows that $\left\Vert T^{n}\right\Vert \rightarrow0,$ and thus the set
$E$ in Lemma \ref{cat} is all of $X$.\bigskip
\end{proof}

The proof of the following theorem is almost exactly the same as the proof of
Theorem 7 in \cite{MV}.

\begin{theorem}
\label{MV}If $X$ is a Banach space and $T\in B\left(  X\right)  $ and $%
{\displaystyle\sum_{n=1}^{\infty}}
\frac{1}{\left\Vert T^{n}\right\Vert }<\infty$, then $T$ is null-orbit
reflexive. If $X$ is a Hilbert space and $%
{\displaystyle\sum_{n=1}^{\infty}}
\frac{1}{\left\Vert T^{n}\right\Vert ^{2}}<\infty$, then $T$ is null-orbit
reflexive. In particular, if $r\left(  T\right)  \neq1,$ then $T$ is
null-orbit reflexive.
\end{theorem}

\begin{corollary}
The set of null-orbit reflexive operators on a Banach space $X$ is norm dense
in $B\left(  X\right)  $.\bigskip\ 
\end{corollary}

\begin{theorem}
\label{polybdd}If $X$ is a Hilbert space and $T\in B\left(  X\right)  $ and is
polynomially bounded, then $T$ is null-orbit reflexive and orbit reflexive.
\end{theorem}

\begin{proof}
We prove the null-orbit reflexivity; the orbit reflexivity is proved in a
similar fashion. Suppose $T$ is polynomially bounded. It was proved by W. Mlak
\cite{Mlak} that $T$ is similar to the direct sum of a unitary operator $U$
and an operator $A$ with a weakly continuous $H^{\infty}$ functional calculus.
In particular, $A^{n}\rightarrow0$ in the weak operator topology. We can
assume $T=U\oplus A.$ We can also assume that the $A$ summand is present;
otherwise, $T$ is null-orbit reflexive by Proposition \ref{nor}. Since
$A^{n}\rightarrow0$ in WOT, it follows from Lemma \ref{cat} that
$\mathrm{nullOrbRef}\left(  A\right)  =\mathrm{nullOrb}\left(  A\right)  $.
Hence we can assume that the $U$ summand is also present. Suppose
$S\in\mathrm{nullOrbRef}\left(  T\right)  $. Then we can write $S=B\oplus C$.
Hence $C\in\mathrm{nullOrb}\left(  A\right)  .$ We also know that
$B\in\mathrm{nullOrbRef}\left(  U\right)  $.

\textbf{Case 1. } $C=0,$ and $B\neq0.$ For a fixed $x_{0}$ with $Bx_{0}\neq0$
and any $y$ there is a sequence $\left\{  n_{k}\right\}  $ of nonnegative
integers such that $\left\Vert T^{n_{k}}\left(  x_{0}\oplus y\right)
-Bx_{0}\oplus0\right\Vert \rightarrow0.$ In particular, $\left\Vert A^{n_{k}%
}y\right\Vert \rightarrow0.$ However, $A^{n}\rightarrow0$ WOT implies there is
an $M>0$ such that $\left\Vert A^{n}\right\Vert <M$ for all $n\geq0$. We want
to show $\left\Vert A^{n}y\right\Vert \rightarrow0.$ Suppose $\varepsilon>0.$
Then there is an $n_{k}$ such that $\left\Vert A^{n_{k}}y\right\Vert
<\varepsilon/M$. If $n\geq n_{k}$, then
\[
\left\Vert A^{n}y\right\Vert \leq\left\Vert A^{n-n_{k}}\right\Vert \left\Vert
A_{n_{k}}y\right\Vert <M\left(  \varepsilon/M\right)  =\varepsilon.
\]
We now know that $A^{n}\rightarrow0$ in the strong operator topology.

Now suppose $m\geq0$ and $A^{m}\neq0.$ Choose $y_{0}$ such that $A^{m}%
y_{0}\neq0.$ For any $x,$ there is a sequence $\left\{  n_{k}\right\}  $ of
integers such that $T^{n_{k}}\left(  x\oplus y_{0}\right)  \rightarrow
S\left(  x\oplus y_{0}\right)  ,$ and it follows that eventually $n_{k}>m.$
Thus, for every $x$ we have $Bx\in\left\{  U^{n}x:n>m\right\}  ,$ so it
follows from Lemma \ref{cat} that $B\in\left\{  U^{n}:n>m\right\}  ^{-SOT}.$
It now follows that there is a net $\left\{  n_{\lambda}\right\}  $ of
positive integers such that $T^{n_{\lambda}}\rightarrow S$ in the strong
operator topology.

\textbf{Case 2.} $C\neq0.$ Since $C\in\mathrm{nullOrb}\left(  A\right)  $,
there is an integer $s\geq0$ such that $C=A^{s}\neq0$. Since $A^{n}%
\rightarrow0$ in the $WOT$, it follows that $Ker\left(  A^{k}-1\right)  =0$
for $k>0$. Thus if $A^{n}y=A^{m}y$ with $n<m,$ then $\left(  A^{m-n}-1\right)
A^{n}y=0,$ which implies that $A^{n}x=0$ and therefore $A^{m}x=0.$ Choose
$y_{1}$ so that $A^{s}y_{1}\neq0$. It follows that if $\left\{  n_{k}\right\}
$ is a sequence of nonnegative integers and $A^{n_{k}}y_{1}\rightarrow
A^{s}y_{1},$ then $n_{k}$ must eventually become $s.$ By considering vectors
of the form $x\oplus y_{1}$, we see that $B=U^{s},$ and therefore $S=T^{s}$.

Since the only remaining case is $S=0\in\mathrm{nullOrb}\left(  T\right)  ,$
the proof is complete.

\bigskip
\end{proof}

\begin{corollary}
If $T$ is a Hilbert space operator and $\left\Vert T\right\Vert \leq1,$ then
$T$ is null-orbit reflexive.
\end{corollary}

\bigskip

\begin{corollary}
If $T$ is a Hilbert space operator with $\left\Vert T\right\Vert =r\left(
T\right)  $ (e.g., $T$ is hyponormal), then $T$ is null-orbit reflexive.
\end{corollary}

\bigskip

The following lemma is a consequence of Theorem \ref{MV}.

\bigskip

\begin{lemma}
\label{last}Suppose $X$ is a Hilbert space, $T\in B\left(  X\right)  ,$
$\lambda\in\mathbb{C}$ with $\left\vert \lambda\right\vert =1.$ If
$\ker\left(  T-\lambda\right)  \neq\ker\left(  T-\lambda\right)  ^{2},$ then
$T$ is null orbit reflexive.
\end{lemma}

\begin{proof}
Suppose $\left\Vert x\right\Vert =1$ and $\left(  T-\lambda\right)  ^{2}x=0$
and $\left(  T-\lambda\right)  x\neq0.$ It follows that
\[
\left\Vert T^{n}x\right\Vert =\left\Vert \left[  \lambda+\left(
T-\lambda\right)  \right]  ^{n}x\right\Vert =\left\Vert \lambda^{n}x+n\left(
T-\lambda\right)  x\right\Vert \geq n\left\Vert \left(  T-\lambda\right)
x\right\Vert -\left\Vert x\right\Vert \geq\varepsilon n
\]
for some $\varepsilon>0$ and for sufficiently large $n$. Thus $%
{\displaystyle\sum}
1/\left\Vert T^{n}\right\Vert ^{2}<\infty$, which, by Theorem \ref{MV},
implies $T$ is null-orbit reflexive.
\end{proof}

\bigskip

\begin{theorem}
Suppose $X$ is a Hilbert space, $T\in B\left(  X\right)  ,$ $r\left(
T\right)  =1$ and no point in $E=\sigma\left(  T\right)  \cap\left\{
z\in\mathbb{C}:\left\vert z\right\vert =1\right\}  $ is a limit point of the
spectrum of $T$. If the restriction of $T$ to the spectral subspace $M_{E}$
for the clopen subset $E$ of $\sigma\left(  T\right)  $ is an algebraic
operator, then $T$ is null-orbit reflexive. In particular, every compact
operator, or algebraic operator on a Hilbert space is null-orbit reflexive.
Hence every operator on a finite-dimensional space is null-orbit reflexive.
\end{theorem}

\begin{proof}
It follows from Lemma \ref{last} that we need only consider the case when
$\ker\left(  T-\lambda\right)  =\ker\left(  T-\lambda\right)  ^{2}$ for every
$\lambda\in E.$ This implies that the restriction of $T$ to $M_{E}$ is similar
to a unitary operator, and since the restriction of $T$ to $M_{\sigma\left(
T\right)  \backslash E}$ has spectral radius less than $1$, we see that $T$ is
similar to a contraction. Hence, by Theorem \ref{polybdd}, $T$ is null-orbit
reflexive. If $T$ is compact or algebraic and $r\left(  T\right)  =1,$ then
the first part of this theorem applies. If $r\left(  T\right)  \neq1,$ then
$T$ is null-orbit reflexive by Theorem \ref{MV}.
\end{proof}

\bigskip

We conclude with another question.

\bigskip

\textbf{Question 4.} Is every power bounded Hilbert space operator orbit
reflexive or null-orbit reflexive?

\end{document}